\documentclass[11pt]{article}

\usepackage{tgpagella}
\usepackage[utf8]{inputenc}
\usepackage[T1]{fontenc}
\usepackage{amssymb}
\usepackage{amsfonts}
\usepackage{amsmath}
\usepackage{amsthm}
\usepackage{mathrsfs}
\usepackage{comment}
\usepackage[normalem]{ulem}
\usepackage{bm}

\newtheorem{theorem}{Theorem}
\newtheorem{proposition}[theorem]{Proposition}
\newtheorem{lemma}[theorem]{Lemma}

\theoremstyle{definition}

\newcommand{\thr}[1]{\mathsf{#1}}

\newcommand{\df}[1]{\textbf{#1}}

\newcommand{\AD}{\thr{AD}}
\newcommand{\DC}{\thr{DC}}
\newcommand{\AC}{\thr{AC}}

\newcommand{\ZF}{\thr{ZF}}
\newcommand{\ZFC}{\thr{ZFC}}
\newcommand{\Det}{\thr{Det}}

\newcommand{\R}{\mathbb{R}}

\renewcommand{\P}{\mathbb{P}}
\newcommand{\J}{\mathbb{J}}
\newcommand{\set}[2]{\left\{ #1 \mid \  #2 \right\}}

\newcommand{\Ord}{\textnormal{Ord}}
\newcommand{\res}{\upharpoonright}
\newcommand{\tuple}[1]{\langle #1 \rangle}
\newcommand{\Aut}{\textnormal{Aut}}
\newcommand{\pow}{\mathcal{P}}

\title{A model with fragments of projective determinacy and failures of $\DC$}
\author{Sandra M\"uller, Bartosz Wcisło}

\begin{document}
	\maketitle
	
	\begin{abstract}
		We describe a construction of a model of second order arithmetic in which (boldface) $\bm{\Pi^1_n}$-determinacy holds, but (lightface) $\Pi^1_{n+2}$-$\DC$ fails, thus showing that no projective level of determinacy implies full $\DC_{\R}$. The construction builds upon the work of Gitman, Friedman, and Kanovei.
	\end{abstract}

	\section{Introduction}
	
	The Axiom of Determinacy is with no doubt one of the most studied principles in the foundations of mathematics. Proposed in \cite{Mycielski_Steinhaus}, it was immediately observed to be incompatible with the Axiom of Choice. Despite of this fact, it produces such a coherent image of what the universe of sets might look like that it is considered to be a serious competitor to its more classical counterpart. Moreover (and perhaps more importantly), under large cardinal assumptions various well-studied classes of ``simple'' sets actually provably satisfy determinacy, so understanding $\AD$ is the key tool in studying these classes. 
	
	Even though $\AD$ is inconsistent with $\AC$, it is consistent with its classical weakening, Dependent Choice, $\DC$. Very often, the applications of choice actually involve just $\DC$, so this assumption is significantly more important for the customary mathematical arguments than the full choice. 
	
	All the current constructions of a model in which $\AD$ holds produce a model in which Dependent Choice for real numbers, $\DC_{\R}$, holds as well (see \cite{DMT}). Therefore, for all what we know $\AD$ can actually imply $\DC_{\R}$ and it is an important open problem whether these two can be separated. 
	
	In this article, we provide a (very) partial solution to this question. We show that if we require that determinacy holds for sets of complexity $\bm{\Pi}^1_n$, $n\geq 1$ this does not entail that $\Pi^1_{n+2}$-$\DC_{\R}$ holds. More specifically, under (a necessary) large cardinal hypothesis, we show that there exists a model $M$ of $\ZF$ satisfying (boldface) $\bm{\Pi}^1_n$-determinacy in which there is a relation $R \subset \R^2$ such that for all $x$, there exists $y$ with $R(x,y)$, but there is no sequence $(x_n)_{n \in \omega}$ such that $R(x_n,x_{n+1})$ holds for all $n \in \omega$. In particular, it shows that the restriction of $\AD$ to any complexity class $\bm{\Pi}^1_n$ does not entail $\DC_{\R}$.
	
	Let us place this result in some context. By \cite{Solovay_AD_DC}, we know that $\AD$ is independent of the full $\DC$, so the question really concerns the dependent choice on $\R$. Moreover, by \cite{Kechris_DC}, if we assume that $\AD$ holds and $V = L(\R)$, this implies full $\DC$. Our result actually concerns the models of this form, so we know that the counterexamples which we produce in general have to include additional data if we want to push it further.

	Now, let us briefly describe the techniques used in our proof. In \cite{Abraham_a_minimal_model} Abraham, following \cite{Jensen_mimal_degree}, introduced a forcing which is itself defined inductively, roughly by adding to some initial very canonical subposet of the Sacks forcing further objects which arise as $L$-least generics for the poset constructed so far in given stage $\alpha$. This can be shown to yield various minimality properties. Most importantly, we can ensure that all reals added by this forcing are actually generic for the forcing. This idea was subsequently adapted by Kanovei in \cite{kanovei79}, and further by Gitman, Friedman, and Kanovei in \cite{fgk} to produce a model in which $\Pi^1_2$-$\DC$ fails, whereas full countable choice and full $\Sigma^1_{\infty}$ comprehension holds. More specifically, we can produce a tree of reals, arising from a generic for a certain inductively defined poset such that the reals in the tree are precisely the ones which are generic for all the posets defined in the stages of the inductive process which, in turn, are definable with a formula of second-order arithmetic.  
	
	This article is based mostly on two observations: first of all, the construction in \cite{fgk}, carried out in $L$, does not really require us to work in that model, but is based on some of its properties shared by many canonical inner models studied in set theory. Second, if the initial model has suitable large cardinal properties, certain amount of determinacy is guaranteed in the resulting model and will pass down to its symmetric extension which is where we actually obtain a failure of projective $\DC_{\R}$. 
	
	The first author gratefully acknowledges that this research was funded in part by the Austrian Science Fund (FWF) [10.55776/Y1498, 10.55776/I6087]. For the purpose of open access, the authors have applied a CC BY public copyright license to any Author Accepted Manuscript version arising from this submission. 
	
	The second author was supported by NAWA Bekker 2022 BPN/BEK/2022/1/00264 project ``Determinacy of games in the arithmetical setting,'' which he acknowledges no less gratefully.

	\section{Preliminaries}
	
	\subsection{Inner Model Theory}
	
	This article will make a crucial use of the notions of \df{premice} and \df{mice}. The definitions of these objects are notoriously technical and demanding, hence we will not give them here, as customary in the field. Standard references for them are \cite{Steel_outline} and \cite{FSIT} (note that the latter source, albeit more detailed, contains a significant error pointed out and corrected in \cite{Schindler_Zeman}), another nice overview can be found in the introduction of \cite{Mueller_thesis}, see also \cite{MSW20}. Roughly, they are models containing large cardinals (most importantly for us, \df{Woodin cardinals}) and satisfying some $L$-like properties. We will explicitly say what we mean by $L$-like in the subsequent sections of this paper.
	
	Other classic notions to which we will refer here are dependent choice for reals, $\DC_{\R}$ and the axiom of determinacy for the projective classes $\bm{\Pi^1_n}$-$\Det$. Crucially for our purposes, certain suitable large cardinal assumptions, roughly the existence of $n$ many Woodin cardinals, guarantee that the model satisfies determinacy for the sets in certain levels of the projective hierarchy, specifically $\bm{\Pi^1_{n+1}}$-sets. Again, these are very well studied and classical notions, to be found, for instance in  \cite{Schindler_book}.
	
	Even though we almost completely omit the presentation of mice, we will mention some of the core notions, which play a key role in this paper. Their precise description can be found in the cited sources (note, however, that some of the terminology below is standard, but not yet established in either \cite{FSIT} or \cite{Steel_outline}). 
	
	If $V$ contains enough large cardinals, then we can build a canonical inner model $L[\vec{E}]$ via a \df{fully backgrounded construction}. Roughly, we build a constructible universe with a sequence of predicates and whenever we see an extender that could fit on a sequence of the current model, we add a predicate denoting it to the sequence and we replace the sequence built thus far with its smallest ``Skolem hull'' (it is  somewhat technical to describe what we really take). If the universe $V$ contains $n$ Woodin cardinals, then $L[\vec{E}]$ also does. Besides that, it satisfies a number of $L$-like properties. Importantly for us, whenever $V[G]$ is a forcing extension by a forcing whose size is below the least measurable cardinal, $L[\vec{E}]^{V[G]} = L[\vec{E}]^V$. The construction is unique in the sense that at each stage $\alpha$, there is the unique way to extend the $\alpha$-th model according to the definition. Like in the case of G\"odel's $L$, we can build a similar model over extra sets obtaining the model $L[\vec{E}](A)$ (which will have analogous properties if the set $A$ in question is small enough). 
	
	Another key notion which we will use are the models $M_n^{\#}(A)$. They are the smallest mice containing $A$ and $n$ Woodin cardinals. In this paper, whenever we write ``$M_n^{\#}(A)$ exists'' we mean that it exists and is $\omega_1$-iterable. These models are canonical in the sense that, if it exists, $M_n^{\#}(A)$ as defined in $V$ is the same as $M_n^{\#}$ as defined in $L[\vec{E}](A)$ (also preserving $\omega_1$-iterability). On the other hand, the existence of the models $M_n^{\#}(x)$ for $x \in \R$ (with the right amount of iterability) is what is actually needed for projective determinacy to hold. 
	
	Suppose that $M$ is a premouse, let $\mathscr{T}$ be an \df{iteration tree} on $M$ of limit length (like with mice and premice, this is another notion whose definition we skip) and let $b$ be a branch through that tree. We define \df{$Q$-structures} of two related sorts: for a tree, $Q(\mathscr{T})$ or for a branch $Q(b, \mathscr{T})$. The first one is the largest extension $\mathscr{Q} \unrhd M(\mathscr{T})$ in which $\delta(\mathscr{T})$ is not definably Woodin (either because, we can find a definable counterexample to $\delta(\mathscr{T})$ being Woodin or the model projects below $\delta(\mathscr{T})$) and which is $\omega_1$-iterable above $\delta(\mathscr{T})$.
	
	The latter is defined similarly, but this is the largest model of the form $M^{\mathscr{T}}_b \mid \gamma$, the limit model of $\mathscr{T}$ along the branch $b$, restricted to the largest ordinal $\gamma$ such that $\delta(\mathscr{T})$ is not Woodin in $M_b^{\mathscr{T}} \mid \gamma$. Crucially, we do not require that the model is $\omega_1$-iterable above $\delta(\mathscr{T})$. In particular, the definition of $Q(b,\mathscr{T})$ has complexity $\Delta_0$. See \cite{Mueller_thesis}, Definition 2.2.1 for the precise formulation.
	
	By a \df{$Q$-structure strategy} for the iteration game  on a premouse $M$, we mean the (possibly partial) strategy $\Sigma$ such that for any tree $\mathscr{T}$ of limit length, $\Sigma(\mathscr{T})$ is the unique branch $b$ of $\mathscr{T}$ for which $Q(b,\mathscr{T}) = Q(\mathscr{T})$. One can show that this in fact gives a unique branch and that the model obtained as a direct limit along this branch is well-founded.
	
	\subsection{Symmetric models}

	In this paper, we will make crucial use of symmetric extensions. Let $M$ be any transitive model of (a fragment of) $\ZFC$ and let $M[G]$ be its forcing extension, where $G$ is generic for a poset $\P$. Recall that a \df{symmetric extension} of $M$ is a model
	\begin{displaymath}
		M \subseteq M^* \subseteq M[G]
	\end{displaymath} 
	such that for some normal subgroup $F \unlhd \Aut(\P)$ of the automorphism group of $\P$, $M^*$ is the set of $x \in M[G]$ such that $x$ has a name $\dot{x}$ which is invariant under the natural action of $F$ on $\P$-names. It is a classical fact that such an $M^*$ is a model of $\ZF$. 
	
	We will extensively use the following simple lemma. The proof is very straightforward, but we include it nevertheless, since we were not able to find an easily accessible reference.
	
	\begin{lemma} \label{lem_definable_subsets_symmetric_ext}
		Let $M^* \subset M[G]$ be a symmetric extension of $M$. Suppose that $X,a \in M^*$ and let $Y$ be a subset of $X$ definable in $M[G]$ with the parameter $a$. Then $Y \in M^*$. 
	\end{lemma}
	
	\begin{proof}
		Let $M^*$ be a symmetric submodel of $M[G]$ induced by the subgroup $F \unlhd \Aut(\P)$. Let $X,a \in M^*$ and let
		\begin{displaymath}
			Y = \set{x \in X}{M[G] \models \phi(x,a)}.
		\end{displaymath}
		Let $\dot{X}$ be a name for $X$ hereditarily invariant under $F$. In particular, for any $\sigma = \tuple{y,p} \in \dot{X}$ and any $f \in F$, $fy = y$ (where $F$ acts on names in the natural way).	Now, consider the following name:
		\begin{displaymath}
			\sigma = \set{\tuple{\dot{x},p}}{\exists q \ \tuple{\dot{x},q} \in \dot{X} \wedge p \Vdash \phi(\dot{x},\dot{a})}.
		\end{displaymath} 
		The name $\sigma$ is clearly hereditarily symmetric. Indeed, if $f \in F$, then 
		\begin{displaymath}
			p \Vdash \phi(\dot{x},\dot{a})
		\end{displaymath}
		iff
		\begin{displaymath}
			f(p) \Vdash \phi(f \dot{x}, f\dot{a}),
		\end{displaymath}
		and by assumption this amounts to
		\begin{displaymath}
			f(p) \Vdash \phi( \dot{x}, \dot{a}),
		\end{displaymath}
		so $\tuple{\dot{x},p} \in \sigma$ iff $\tuple{\dot{x},f(p)} \in \sigma$.
		
		Now, we claim that $\sigma^G = Y$. Indeed, let $x \in \sigma^G$ and assume that $\tuple{\dot{x},p} \in \sigma$. Then 
		\begin{displaymath}
			p \Vdash \phi(\dot{x},\dot{a}),
		\end{displaymath} 
		so by definition $x \in Y$. 
		
		If, on the other hand, $x \in Y$, then there exists a symmetric name $\dot{x}$  for $x$ and $p \in G$ such that $p \Vdash \phi(\dot{x},\dot{a}).$ Hence, $\tuple{\dot{x},p} \in \sigma$ and $x \in \sigma^G$. 
		
	\end{proof}

	\section{Adapting Gitman--Friedman--Kanovei poset} \label{sec_adapting_gkf}
	
	Our work builds heavily on \cite{fgk}, which described a construction of a model in which $\Pi^1_2\text{-}\DC$ fails (which in turn, was building upon  previous work by Jensen and Abraham). The model was obtained as a symmetric extension:
	\begin{displaymath}
		L \subset N \subset L[G],
	\end{displaymath}
	where $L$ is the constructible universe and $G$ is a generic for a certain poset $\J$ constructed using specific properties of $L$. In $N$, there is a tree $T$ of reals which has no branch. This much can be obtained by much simpler forcing, not utilising the technology originally studied by Jensen. However, the following two properties hold:
	\begin{itemize}
		\item The tree is exactly the tree of generic reals in $L[G]$.
		\item Being generic for the poset $\J$ can be defined in a $\Pi^1_2$-manner. 	
	\end{itemize}
	The above facts use the exact manner in which our forcing $\J$ is constructed, which in turn invokes specific properties of $L$. However, even though the construction of $\J$ as described in their paper in many places uses the properties of the constructible universe $L$, it can be adapted to other models of set theory, provided they satisfy enough $L$-like properties. We will now isolate specific assumptions we are making about the ground model and give an outline of the construction of $\J$ indicating which parts may be treated as a blackbox. Suppose $M$ is a model of $\ZFC^-$ such that:
	\begin{itemize}
		\item There exists a $\Diamond$-sequence definable via a formula $D$.
		\item There exists a well-ordering of the reals definable via a formula $W$. 
		\item There exists a formula $K$ such that for any set of ordinals $x$ if $K(x)$ holds, then $x$ codes a well-founded model $\bar{M}$ of $\ZFC^-$ such that:
		\begin{itemize}
			\item  $D(\bar{M})$ is a $\Diamond$-sequence and an initial segment of the $\Diamond$-sequence defined by $D$ in $M$;
			\item likewise, $W(\bar{M})$ is a well-ordering of the reals of $\bar{M}$ which is an initial segment of the well-ordering defined by $W$ in $M$. 
		\end{itemize}
		\item The formula $K$ holds for a club of countable subsets of $\omega_1$.\footnote{Actually, in Section 5, we will arrange the situation slightly differently: there will be a formula $K_0$ which will compute the well-ordering and which will be used for the complexity calculations, and another formula $K(x)$ which will say ``every real in the model coded by $x$ is in the model satisfying $K_0(x)$'' and which will ostensibly hold for a club of models.}
	\end{itemize}
	
	If a model $M$ satisfies the above properties, let us call it \df{$L$-like}. Let us now describe the construction. 
	
	Working in $M$, we will construct a sequence of posets $\P_{\alpha}, \alpha<\omega_1$. Each of them will be a \df{perfect poset}, i.e., a subposet of the Sacks forcing satisfying the following conditions:
	\begin{itemize}
		\item It contains all trees $T_s : = \set{x \in \omega^{\omega}}{s \textnormal{ is a prefix of } x}$, where $s \in \omega^{<\omega}$. 
		\item It is closed under joins $T \wedge S$, where $T \wedge S$ is the maximal perfect tree in $T \cap S$, if such a tree exists, or an empty tree otherwise. 
		\item It is closed under unions.
	\end{itemize} 
	The construction itself uses objects which we will not define, but the definitions of which, crucially, do not depend on the ground universe $L$. Specifically, these notions are:
	\begin{itemize}
		\item A \df{tree iteration of perfect posets}, $\mathbb{P}(\vec{P}, T)$, where $\vec{P}$ is (some) $\omega$-iteration of perfect posets and $T$ is a tree of height $\omega$. (Definition 6.1 of \cite{fgk}).
		\item A \df{fusion poset} $\mathbb{Q}(\vec{P},T)$, where $\vec{P}$ and $T$ are as above. (Remarks after Definition 6.1 of \cite{fgk}).
		\item For any generic $H \subset \mathbb{Q}(\vec{P},T)$, a \df{seal poset} $\P^* = S(\P,H)$. (Remarks after Proposition 6.5 of \cite{fgk}).
	\end{itemize}
	
	In the tree iteration, the conditions are roughly finite subtrees of $T$ where the nodes are conditions come from $\vec{P}$, i.e., from the usual finite support length $\omega$ iterations of $\P$ and we require them to be arranged in a coherent manner. A condition $p$ in this forcing is stronger than a condition $q$, essentially, if the finite tree supporting the iteration is bigger and the conditions are pointwise stronger. 
	
	The fusion poset for Sacks forcing is a poset where conditions are pairs $(T,n)$ where $T$ is a perfect tree and $n$ is a natural number. If $(T,n) \leq (S,k)$, then $T \subseteq S, k \geq n$, and, crucially, we require that $T \res n = S \res n$. In that manner, we guarantee that our forcing adds a generic \emph{tree} rather than a generic real. This idea can be naturally carried over to iterations and to tree iterations, but the details are somewhat technical.  
	
	The name ``seal poset'' is introduced in this note and seems to have no previous established name. The seal poset takes a poset $\P$, its finite support length $\omega$ iteration $\vec{P}$, and a generic $H$ for the fusion poset $\mathbb{Q}(\vec{P},T)$. This generic adds a tree of perfect trees. We take these trees and generate, in a natural manner, a minimal perfect poset extending $\P$ and containing all of them. Crucially this construction depends only on a poset $\vec{P}$ and a generic $H$ for the fusion poset  $\mathbb{Q}(\vec{P},T)$. Neither the construction, nor the basic facts concerning it really make use of the global properties of the background universe.
	
	The name is justified by the following result:
	
	\begin{lemma} \label{lem_sealing_posets} \
		Let $\P$ be a perfect poset and let $\P^* = \mathbb{Q}(\P,\omega_1^{<\omega})$. 
		\begin{itemize}
			\item Let $\mathscr{A}$ be a maximal antichain for an $n$-fold iteration $\P_n$ of $\P$. Then it is a maximal antichain for the $n$-fold iteration $\P_n^*$ of $\P^*$.
			\item Let $\mathscr{A}$ be a maximal antichain for the tree iteration $\mathbb{P}(\mathbb{P},\omega_1^{<\omega})$. Then it is a maximal antichain for $\mathbb{P}(\mathbb{P}^*,\omega_1^{<\omega}).$
		\end{itemize}
	\end{lemma}
	
	Let us now describe the actual construction: we start with a poset $\P_0$ which consists of all trees $T_s, s \in \omega^{< \omega}$. At the stage $\alpha$, we are given a perfect poset $\P_{\alpha}$. We let $\P_{\alpha+1} = \P_{\alpha}$ unless in the $\Diamond$-sequence $(D_{\xi})_{\xi < \omega_1}$ defined by $D$:
	\begin{itemize}
		\item  $K(D_{\alpha})$ holds, so $D_{\alpha}$ codes a well-founded $L$-like model $M_{\alpha}$ resembling $M$;
		\item $P_{\alpha} \in M_{\alpha}$.
		\item $\omega_1^{D_{\alpha}}= \alpha$.
	\end{itemize}
	In this case, we take $G_{\alpha}$ to be the $W$-least generic for the fusion tree poset $\mathbb{Q}(\vec{P},\omega_1^{<\omega})$ and we take $\P_{\alpha+1}$ to be the seal poset $S(\P_{\alpha},G_{\alpha})$.
	
	Finally, we will denote:
	\begin{displaymath}
		\J = \bigcup_{\alpha <\omega_1} \P_{\alpha}.
	\end{displaymath}
	The exact definition of $\J$ depends of course on the formulae $D,W,K$, so it really should be parametrized. However, we will suppress the explicit mention of the definitions. We obtain a series of results regarding $\J$ as in \cite{fgk}.
	
	\begin{lemma} \label{lem_J_has_ccc}
		The poset $\J$ and the poset $\mathbb{P}(\mathbb{J},\omega_1^{<\omega})$ have the ccc property.
	\end{lemma}
	
	The next lemma was a part of the proof of the main result in \cite{fgk}. Let us isolate it as a separate fact. 
	
	\begin{lemma} \label{lem_generic_is_locally_generic}
		Let $G$ be generic for $\mathbb{P}(\mathbb{J},\omega_1^{<\omega})$. Let $\bar{x} = \tuple{x_1, \ldots, x_n} \in M[G]$ be an $n$-tuple of reals. Then $\bar{x}$ is generic for the $n$-fold iterate $\J_n$ iff for any $\alpha < \omega_1$ it is generic for  $\mathbb{P}(\mathbb{P}_{\alpha},\omega_1^{<\omega})$.
	\end{lemma}
	
	\begin{proof}
		
		$(\Rightarrow)$. Easy.
		
		$(\Leftarrow)$. Fix any $\bar{x} \in M[G]$ and suppose that it is generic for any forcing $\mathbb{P}(\mathbb{P}_{\alpha},\omega_1^{<\omega})$. Fix an arbitrary maximal antichain $\mathscr{A} \subset \mathbb{P}(\J,\omega_1^{<\omega})$. By Lemma \ref{lem_J_has_ccc}, it is countable, so there exists a stage $\alpha < \omega$ such that $\mathscr{A} \subset \mathbb{P}(\P_{\alpha},\omega_1^{<\omega})$. By assumption, $\bar{x}$ meets this antichain. 
	\end{proof}
	
	\begin{lemma}[Kanovei--Lyubetsky Theorem for $\J$] \label{lem_kanovei_lyubetsy}
		Let $\P(\J,\omega_1^{<\omega})$ be the tree iteration poset, let $G \subset \P(\J,\omega_1^{<\omega})$ be a generic filter, and let $\set{x_s}{s \in \omega_1^{<\omega}}$ be the tree of reals naturally arising from the generic $G$. Then for any tuple $\bar{x} \in {(\R^{n})}^{M[G]}$ which is generic for $n$-iteration of $\J$, there exists a tuple $s \in \omega_1^{n}$ for which $\bar{x} = x_s$. 
	\end{lemma}

	\section{A model with failures of $\DC$ and $\Pi^1_1$-determinacy}
	\label{sec_model_with_measurable_and_a_failure_of_DCR}
	
	In this subsection, we show that the construction of the symmetric model from \cite{fgk} can be carried over to the $L[U]$ setting, resulting in a model with definable failures of $\DC$. Here, there are some differences between our construction and the original one, since dealing with models with a measurable raises the complexity of the defined tree. 
	
	Let us introduce some notation. Let $F \leq \Aut(\omega_1^{< \omega})$ be a subgroup consisting of the order-automorphisms of $\omega_1^{<\omega}$ as a tree which pointwise fix some countable subtree $T \subset \omega_1^{\omega}$ with no infinite branch. This subgroup acts naturally on the poset $\P(\J,\omega_1^{<\omega})$ by permuting the names on the coordinates. 
	
	Let $G \subset \P(\J,\omega_1^{< \omega})$ be a generic filter and let $N$ be a \df{symmetric model} induced by the permutation subgroup $F$, i.e.:
	\begin{displaymath}
		L[U] \subset N \subset L[U][G]
	\end{displaymath}
	such that 
	\begin{displaymath}
		N = \set{x \in L[U][G]}{x = \dot{x}^G \textnormal{, where $\dot{x}$ is hereditarily invariant under }F}.
	\end{displaymath}
	We want to show that in $N$, there exists a definable tree of reals which has no branch. 
	
	Let $T$ be the tree of $\J_n$-generic tuples of reals added by the forcing $\P(\J,\omega_1^{<\omega})$. Notice that all such tuples from $L[U][G]$ are in $N$, since a condition from $\P(\J,\omega_1^{<\omega})$ is supported by a finite tree $X$, so it is fixed by any automorphisms fixing a countable tree with no branch extending $X$. Since $\dot{T}$ contains all such conditions, it is clearly fixed setwise by $F$ (and, in fact, by $\Aut(\omega_1^{<\omega}))$. It can be also shown that no branch through $\omega_1^{<\omega}$ can have a name fixed by all automorphisms fixing a countable tree $S$ with no branches --- essentially because we can change that branch at nodes which are outside $S$ without changing $S$. This locality property was established in Lemma 9.3 of \cite{fgk}. Finally, we have to check the complexity of the branch defined in this way. 
	
	\begin{proposition} \label{prop_complexity_of_tree_Pi13}
		Let $N,T$ be defined as above. Then $T$ can be defined as a $\Pi^1_2$ subset of the reals.
	\end{proposition}
	\begin{proof}
		By Lemma \ref{lem_kanovei_lyubetsy}, the tuples of reals in the tree $T$ are exactly the tuples which are $\J_n$-generic over $L[U]$, so it is enough to show that the latter condition can be expressed in $\Pi^1_2$ manner. 
		
		Consider the following statement $\phi(x)$:
		\begin{center}
			''For all well-founded models $M$ satisfying $V = L[U]$, if all countable iterates of $M$ via $U$ are well founded, then for any sequence of sets $(M_{\beta})_{\beta < {\omega_1}^M}$, at any nontrivial stage of the construction, $x$ defines a $\P_{\beta}$- generic filter.''
		\end{center}
		This formula defines the tuples $x$ which are $\J_n$-generic over $L[U]$. Now it is enough to check its complexity. Notice that the formula has the form:
		\begin{multline*}
			\underbrace{\textnormal{for all $X$}}_{\forall} \Big( \underbrace{\textnormal{$X$ codes a well founded model whose all iterates are wellfounded }}_{\forall} \\
			\rightarrow \underbrace{\textnormal{there is a (unique) construction of $\J_n$ which makes $x$ generic}}_{\exists}\Big).
		\end{multline*}
		
		which is a equivalent to a formula of the form:
		\begin{displaymath}
			\forall(\exists  \vee \exists),
		\end{displaymath}
		so it has complexity $\Pi^1_2$, as required.

	\end{proof}

	\subsection{A measurable cardinal implies determinacy in the symmetric model} \label{ssec_pi11det_from_measurable}
	
	In the previous section, we have constructed a symmetric extension of $L[U]$ in which $\Pi^1_2$-$\DC$ fails. Now, we will check that (boldface) $\bm{\Pi}^1_1$-determinacy holds in the model. The arguments in this subsection could be replaced by slightly more abstract ones, along the lines of Section \ref{ssec_projective_det_in_small_forcing_extensions}. However, since the proof of $\bm{\Pi^1_1}$-determinacy is significantly easier than the proof of projective determinacy, it might be actually instructive to see this case spelled out in a more explicit manner. Therefore, we include the argument for the convenience of the reader.
	
	\begin{theorem} \label{det_in_symm_models}
		Suppose that $\kappa$ is a measurable cardinal. Let $N$ be a symmetric extension obtained by forcing $\P$ with a group $F$. Then $N$ satisfies $\mathbf{\Pi}^1_1$-determinacy
	\end{theorem}
	The proof of this theorem follows the original argument by Martin. We start with a combinatorial Lemma (the proof, in the context of $\ZFC$, can be found for instance in the proof of Martin's result, Theorem 31.1 in \cite{kanamori}):
	
	\begin{lemma}[Order representation for $\Pi^1_1$ sets] \label{lem_order representation}
		Let $A \subset \omega^{\omega}$. Suppose that there exists a tree $T$ such that $A \subseteq [T]$. Then $A$ is $\mathbf{\Pi}^1_1$ iff there exists a function $ T \ni t \mapsto <_t$ such that the following conditions are satisfied:
		\begin{itemize}
			\item For all $t \in T$, $<_t$ is a linear order on $\{0,1, \ldots, |t|\}.$
			\item For all $t \subset s$, $<_t \subset <_s$. 
			\item For all $x \in \omega^{\omega}$, $x \in A$ iff $<_x$ is wellfounded, where $<_x = \bigcup_{n \in \omega} <_{x \res n}$.
		\end{itemize} 
	\end{lemma}
	The ordering from the above lemma is produced in a constructive manner and can be carried out in $\ZF$. Also, the following classic fact can be proved without using choice (a proof in $\ZFC$ can be found, for instance, as Proposition 27.1 in \cite{kanamori}):
	
	\begin{theorem}[Closed determinacy] \label{th_closed_determinacy_without_ac}
		Let $X$ be a wellfounded set. Let $A$ be a closed subset of $X^{\omega}$ in the product topology with $X$ discrete. Then the Gale-Stewart game on $A$ is determined. 
	\end{theorem}
	The above theorem probably cannot be prove in a choiceless context for general sets $X$. The point is that using the well order on $X$, in each move of the game, we can pick the smallest position in which we have not yet lost.

	Now, in order to prove $\bm{\Pi}^1_1$-determinacy, it is enough to reduce $\bm{\Pi}^1_1$-games to closed games. Let $\kappa$ be a measurable cardinal (in $V$). Let $A \subset \omega^{\omega}$ be a $\bm{\Pi}^1_1$ set. Let us consider the following game: Player I plays pairs of the form $\tuple{n_i,\xi_i} \in \omega \times \kappa$. Player II plays natural numbers $n_j$. I wins if the produced real $x$ is in $A$, and, additionally, for any $k \in \omega$, and any $i,j < \frac{k}{2}$ the following condition is satisfied:
	\begin{displaymath}
		n_i <_{x \res k} n_j \textnormal{ iff } \xi_j < \xi_j. 
	\end{displaymath}
	In other words, I has to produce a number in $A$, and additionally witness that $x \in A$ by embedding the linear order $<_x$ in $\kappa$. Since this is a closed game, it is determined. A winning strategy for I in this auxiliary game clearly produces a winning strategy for I in the original game. So it is enough to check that a winning strategy for II in this game yields a winning strategy (for II) in the original game. 
	
	By assumption, $\kappa$ is a measurable cardinal. Suppose that $U$ is a $\kappa$-complete ultrafilter on $\kappa$ in $V$. Then 
	\begin{displaymath}
		\widehat{U} = \set{X \in \pow(\kappa)^N}{\exists X_0 \in U \ X_0 \subset X}
	\end{displaymath}
	is an ultrafilter in $N$. Indeed: take any $Y \in \pow(\kappa)^N$. Fix a good name $\dot{Y}$ for $Y$. Let 
	\begin{displaymath}
		Z = \set{\gamma \in \kappa}{\exists p \in \P \ \tuple{\gamma,p} \in \dot{Y}}.
	\end{displaymath}
	If $Z \notin U$, then very few ordinals even have a chance to end up in $Z$ and 
	\begin{displaymath}
		\kappa \setminus Z \subseteq \kappa \setminus Y \in \widehat{U}.
	\end{displaymath}
	If, on the other hand, $Z \in U$, then since $\P$ is a small forcing, we see that there exists a single $p$ such that 
	\begin{displaymath}
		Y_0 = \set{\gamma \in \kappa}{\tuple{\gamma,p} \in \dot{Z}} \in U.
	\end{displaymath}
	Then 
	\begin{displaymath}
		Y_0 \subseteq Y \in \widehat{U}.
	\end{displaymath}
	This proves our claim. Now, let $U$ be a normal measure on $\kappa$ in $V$. Let
	\begin{displaymath}
		U_2 = \set{X \times X}{X \in U}.
	\end{displaymath}
	The set $U_2$ exists both in $V$ and in $N$ and it is known to be an ultrafilter on $[\kappa]^2$ in $V$ and, by the previous argument, also an ultrafilter in $N$. We define $U_n$ for $n \in \omega$ in a similar manner. 
	
	Now, suppose that $\sigma$ is a winning strategy for II in the auxiliary game. We will define a winning strategy $\tau$ for II in the original game in the following manner: by $\omega$-completeness of $U$, for any tuple $n_0, \ldots, n_{2k}$, there exists a unique $m$ such that:
	
	\begin{displaymath}
		A_m = \set{\tuple{\xi_0, \ldots, \xi_k}}{\tau(n_0,\xi_0,n_1,\ldots,n_{2k},\xi_k) = m} \in U_k.
	\end{displaymath}
	We let 
	\begin{displaymath}
		\tau(n_0,n_1,\ldots,n_{2k}) = m.
	\end{displaymath}
	That $\tau$ indeed defines a winning strategy can be checked as in the original proof (see Theorem 31.1 in \cite{kanamori}).
	
	\section{Definable failures of $\DC$ in models with Woodin cardinals}

	Now we will extend the results from the previous sections to the context of higher levels of projective determinacy. The overall flavour of the argument will be very similar to the previous one. However, there will be an extra layer of technical detail to take care of. 
	
	We will be working in the canonical inner model $M_n$. Again, in what follows we will not rely on a specific construction, but we will instead need to ensure that it is $L$-like in the sense of Section \ref{sec_adapting_gkf} and has large cardinal properties. The precise statements we need have been worked out in \cite{Steel_projectively_well_ordered}:
	
	\begin{itemize}
		\item If there exist at least $n$ Woodin cardinals, then the model $M_n$ exists.
		\item $M_n \models $ ``there exist $n$ distinct Woodin cardinals.''
		
		\item In $M_n$, there exists a $\Delta_{n+2}$-definable well-ordering of $\R$ defined with a formula $W(x)$.
		\item In $M_n$, there exists $\Diamond$-sequence $(S_{\alpha})_{\alpha < \omega_1}$ defined with a formula $D$.
		\item Let $K_0(x)$ mean: ``$x$ codes a well-founded model $\bar{M}$ which is an $n$-small, $\Pi_n$-iterable $\omega$-mouse.'' 
		\item Let $K(x)$ mean: ``$x$ codes a well founded model $\bar{N}$ such that every real of $\bar{N}$ is in the model coded by a real $y$ satifying $K_0(y)$.''\footnote{See Definitions 1.1, 1.3, 1.4, and 1.6 in \cite{Steel_projectively_well_ordered}. A mouse is $n$-small if it does not have $n$ Woodin cardinals; $\omega$-mouse is, roughly, a mouse projecting to $\omega$ and universal, $\Pi_n$-iterability is defined by the existence of winning strategies in certain more carefully defined iteration games; the point is that assuming $\Delta^1_{n+1}$-detereminacy (required only for even $n$), this condition is $\Pi^1_{n+1}$. In the next section, we will see that the suitable amount of determinacy holds in $M^*$. The models satisfying $K(x)$ form a club, since elementary submodels of large initial segments of $M_n$ satisfy it.} Then for any transitive set $\bar{M}$ coded by $x$ satisfying $K(x)$, $W^{\bar{M}}$ is an initial segment of $W^{M_n}$. 
		\item Likewise, if $K(x)$ holds and $x$ codes a model $\bar{M}$, then $D^{\bar{M}}$ is an initial segment of $D^{M_n}$.
	\end{itemize}
	
	We define the model $N$ in the same manner as in Section \ref{sec_model_with_measurable_and_a_failure_of_DCR}, as the symmetric model obtained by considering the names which are stabilised under the subgroup of automorphisms of ${\omega_1}^{<\omega}$ fixing a countable subtree with no branch. We can still show that the tree $T$ naturally obtained from the generic $G$ for $\J$ is an ${\omega_1}^{<\omega}$ tree with no branch. 
	
	Now, crucially, we have to prove two facts:
	
	\begin{itemize}
		\item If $M$ is a model with $n$ Woodin cardinals $\delta_1 < \ldots < \delta_n$, and $G$ is a generic for a small forcing $\P$, then in any model $N \models \ZF$ with $M \subset N \subset M[G]$, $\bm{\Pi}^1_{n+1}$-determinacy holds. 
		\item The tree $T$ obtained from a generic for $\J$ defined as above in $M_n$ is $\Pi^1_{n+2}$-definable.
	\end{itemize}
	
	The first of the above facts will be proved in Subsection \ref{ssec_projective_det_in_small_forcing_extensions}, so let us take care of the second issue.
	
	\begin{proposition} \label{prop_symmetric_models_easily_defines_the_tree}
		In the model $N$, the tree $T$ is $\Pi^1_{n+2}$-definable.
	\end{proposition}
	
	\begin{proof}
		By Lemma \ref{lem_kanovei_lyubetsy}, it is enough to define in $N$ the set of tuples $\bar{x}$ which are $\J_n$-generic. By Lemma \ref{lem_generic_is_locally_generic}, we need to check whether $\bar{x}$ is generic for each poset $\mathbb{P}(\mathbb{P}_{\alpha},(\omega_1^{<\omega})^{M_{\alpha}})$, where $K(M_{\alpha})$ holds. 
		
		However, the latter condition can be expressed with a formula:
		\begin{multline*}
			\forall X, Y \ \Big(\underbrace{K_0(X)}_{\Pi^1_{n+1}} \wedge \underbrace{\textnormal{$Y$ is the result of the $\J$-construction in $X$ }}_{\Sigma^0_{\infty}} \\
			\rightarrow \underbrace{\bar{x} \textnormal{ meets all maximal antichains of }Y}_{\Pi^1_1} \Big). 
		\end{multline*} 
		Therefore, this is a formula of the form:
		\begin{displaymath}
			\forall Z (\Pi^1_{n+1} \rightarrow \Pi^1_1 ),
		\end{displaymath}
		so it is $\Pi^1_{n+2}$.
	\end{proof}
	
	\subsection{Projective determinacy in the small forcing extensions of $M_n$} \label{ssec_projective_det_in_small_forcing_extensions}


	In order to complete the proof, we still have to verify whether in the model $N$ constructed in the previous section, $\bm{\Pi}^1_{n+1}$-determinacy still holds. In this section, we will verify this fact. Unlike in Subsection \ref{ssec_pi11det_from_measurable}, our argument will not follow the original determinacy proof directly, since it is much more technically involved. Instead, we will show how the proof can be used as a blackbox, using classical techniques from inner model theory. We will prove the following Proposition:
	
	\begin{proposition} \label{prop_projective_detereminacy_in_intermediate_models}
		Suppose that $M$ is a model of $\ZFC$ with $n$ Woodin cardinals. Let $\P$ be a forcing poset with $|\P|$ smaller than the least measurable. Let $M \subseteq M^* \subseteq M[G]$, where $M^*$ is a symmetric extension and $G$ is a generic for $\P$. Then in $M^*$ $\bm{\Pi}^1_n$-determinacy holds. 
	\end{proposition}
	
	\begin{theorem}[Neeman] \label{th_determinacy_from_mice}
		Suppose that $M$ is a model of $\ZF$ which is closed under the $M_n^{\#}$ operator. Then $\bm{\Pi}^1_{n+1}$-determinacy holds in $M$. 
	\end{theorem}
	
	Unfortunately, Theorem 2.14 in \cite{Neeman_optimal2} states it in the context of $\ZFC$. However, the result actually holds for $\ZF$ (for instance, it is stated in pure $\ZF$ context in \cite{Busche_Schindler}, Theorem 3.3). The proof of  Proposition \ref{prop_projective_detereminacy_in_intermediate_models}, uses the fact that forcing extensions preserve sharps (and, in fact, even the $M_n^{\#}$-operator). A proof can be found, for instance, in \cite{Busche_Schindler}, Lemma 3.7.
	
	\begin{lemma} \label{lem_sharps_in_forcing_extensions}
		Let $M$ be an arbitrary model of $\ZFC$. Suppose that 
		\begin{displaymath}
			M \models \textnormal{For any $A \subset \Ord$, $M_n^{\#}(A)$ exists}.
		\end{displaymath}
		Then for any $G \subset \P$ generic,
		\begin{displaymath}
			M[G] \models \textnormal{For any $A \subset \Ord$, $M_n^{\#}(A)$ exists}.
		\end{displaymath}
	\end{lemma}

	Now, we will prove our proposition. 
	\begin{proof}[Proof of Proposition \ref{prop_projective_detereminacy_in_intermediate_models}] \
		
		Let $\delta_0<\ldots<\delta_{n-1}$ be the Woodin cardinals of $M$. Since the forcing is small, they are also Woodin in $M[G]$.
		\paragraph*{Claim I}
		We will show that for every $X \in M^*_{\delta_0}(:= M^* \cap V_{\delta_0})$, $ X \subseteq \Ord$,
		\begin{displaymath}
			M^*_{\delta_0} \models \textnormal{ There exists } M_{n-1}(X)^{\#}.
		\end{displaymath}
		First, notice that since the forcing is small, there exists $M_{n-1}(X)^{\#}$ in $M[G]_{\delta_0}.$ 
		
		Within $M[G]$, define:
		\begin{displaymath}
			F : = \set{\tuple{A,B}}{L[\vec{E}](A) \models B = A^{\#}},
		\end{displaymath}
		where $L[\vec{E}](A)$ is the fully backgrounded construction in $M[G]$ above $A$. Notice that the structure $L[\vec{E}](A)$ can be defined correctly in $M[G]$. In the model $M[G]$, the extenders are lift-ups of the extenders from $M$. By Theorem 5.1 of \cite{Schlutzenberg_uniqueness}, we know that running a backgrounded construction of $L[\vec{E}]$ will yield a unique partial extender at each step of the construction (actually, Theorem 9.1 of \cite{FSIT} is sufficient for this argument, as we can mimic the choices of either type I, II, III extenders made in $M$ throughout the whole construction). Therefore $L[\vec{E}](A)$ will be extended with a (globally) definable subset at each active step. Using inductively Lemma \ref{lem_definable_subsets_symmetric_ext}, we obtain $(L[\vec{E}](A))^{M*}=(L[\vec{E}](A))^{M[G]}$.

		Consider now $F' \in M^*$ defined in the same way as $F$. Notice that since the construction of $L[\vec{E}](A)$ is absolute between $M[G]$ and $M^*$, we have that for $A \in M^*$, $\tuple{A,B} \in F$ iff $\tuple{A,B} \in F'$. However, since $M[G]$ is a $\ZFC$ model, we can see that:
		\begin{displaymath}
			M[G] \models \textnormal{If } \tuple{A,B} \in F \textnormal{, then } B = M_{n-1}(A)^{\#}.  
		\end{displaymath}
		This finishes the proof of Claim I.
		
		\paragraph*{Claim II} Let $X \subseteq \Ord$ be an element of $M^*_{\delta_0}$ and let $k \leq n$. Let $Y = (M_{k}(X)^\#)^{M[G]}.$ Then
		\begin{displaymath}
			M^*_{\delta_0} \models Y = M_k(X)^{\#}.
		\end{displaymath}
		
		Fix $X$ and let $N = M_k(X)^{\#}$, as defined in $M_{\delta_0}[G]$. The definition of an $n$-small premouse is clearly absolute, so we only have to check whether $N$ is $\omega_1$-iterable in $M^*$ and whether it is the smallest such mouse (notice that \textit{prima facie}, a structure which is not iterable in $M[G]$ could be iterable in $M^*$ and \textit{vice versa}, since there are less trees to deal with, but also less branches to respond with).
		
		We will show that $N$ is $\omega_1$-iterable in $M^*$. By induction on $\alpha$, we show that if $\mathscr{T} \in M^*$ is an iteration tree on $N$ of a limit length $\alpha$, then the branch $b$, given by the $Q$-structure strategy is in $M^*$ (and, since $M^* \subset M[G]$, the model $M^{\mathscr{T}}_b$ is wellfounded). The induction step follows by Lemma \ref{lem_definable_subsets_symmetric_ext}, since the branch $b$ given by the $Q$-structure strategy is unique. 
		
		Wellfoundedness of the model computed by the $Q$-structure strategy in $M^*$ is obvious since $M^* \subset M[G]$. Therefore, by Theorem \ref{th_determinacy_from_mice}, $M^*_{\delta_0}$ satisfies $\bm{\Pi}^1_n$ -determinacy, and so does $M^*$.

	\end{proof}


\begin{thebibliography}{10}
		
		\bibitem{Abraham_a_minimal_model}
		Uri Abraham.
		\newblock A mimimal model for $\neg$ {CH}: Iteration of {J}ensen's reals.
		\newblock 281(2), 1984.
		
		\bibitem{Busche_Schindler}
		Daniel Busche and Ralf Schindler.
		\newblock The strength of choiceless patterns of singular and weakly compact
		cardinals.
		\newblock {\em Annals of Pure and Applied Logic}, 159(1):198--248, 2009.
		
		\bibitem{fgk}
		Sy-David Friedman, Victoria Gitman, and Vladimir Kanovei.
		\newblock A model of second-order arithmetic satisfying {AC} but not {DC}.
		\newblock {\em Journal of Mathematical Logic}, 19(01):1850013, 2019.
		
		\bibitem{Jensen_mimal_degree}
		Ronald Jensen.
		\newblock Definable sets of minimal degree.
		\newblock In Yehoshua Bar-Hillel, editor, {\em Mathematical Logic and
			Foundations of Set Theory}, volume~59 of {\em Studies in Logic and the
			Foundations of Mathematics}, pages 122--128. Elsevier, 1970.
		
		\bibitem{kanamori}
		Akihiro Kanamori.
		\newblock {\em The Higher Infinite: Large Cardinals in Set Theory from Their
			Beginnings}.
		\newblock {S}pringer-{V}erlag {B}erlin, 2008.
		
		\bibitem{kanovei79}
		Vladimit Kanovei.
		\newblock On descriptive forms of the countable axiom of choice.
		\newblock In {\em Studies in nonclassical logic and set theory}. 1979.
		
		\bibitem{Kechris_DC}
		Alexander~S. Kechris.
		\newblock The axiom of determinancy implies dependent choices in
		{L($\mathbb{R})$}.
		\newblock {\em The Journal of Symbolic Logic}, 49(1):161--173, 1984.
		
		\bibitem{FSIT}
		William~J. Mitchell and John~R. Steel.
		\newblock {\em Fine Structure and Iteration Trees}.
		\newblock Lecture Notes in Logic. Cambridge University Press, 2017.
		
		\bibitem{Mueller_thesis}
		Sandra M{\"u}ller.
		\newblock Pure and hybrid mice with finitely many woodin cardinals from levels
		of determinacy, {PhD} thesis.
		
		\bibitem{Mycielski_Steinhaus}
		Jan Mycielski and Hugo Steinhaus.
		\newblock A mathematical axiom contradicting the {A}xiom of {C}hoice.
		\newblock {\em Bulletin de l'Académie Polonaise des Sciences, Série des
			sciences mathématiques, astronomiques et physiques}, 10:1--3, 1962.
		
		\bibitem{MSW20}
		S.~Müller, R.~Schindler, and W.H. Woodin.
		\newblock Mice with finitely many {W}oodin cardinals from optimal determinacy
		hypotheses.
		\newblock {\em Journal of Mathematical Logic}, 20, 2020.
		
		\bibitem{Neeman_optimal2}
		Itay Neemna.
		\newblock Optimal proofs of {D}eterminacy {II}.
		\newblock {\em Journal of Mathematical Logic}, 02(02):227--258, 2002.
		
		\bibitem{Schindler_book}
		Ralf Schindler.
		\newblock Set theory: Exploring independence and truth.
		\newblock 2014.
		
		\bibitem{Schindler_Zeman}
		Ralf-Dieter Schindler, John Steel, and Martin Zeman.
		\newblock Deconstructing inner model theory.
		\newblock {\em The Journal of Symbolic Logic}, 67(2):721--736, 2002.
		
		\bibitem{Schlutzenberg_uniqueness}
		Farmer Schlutzenberg.
		\newblock The definability of {$\vec{E}$} in self-iterable mice.
		\newblock {\em Ann. Pure Appl. Log.}, 174:103208, 2014.
		
		\bibitem{Solovay_AD_DC}
		Robert~M. Solovay.
		\newblock {\em The independence of {DC} from {AD}}, page 66–95.
		\newblock Lecture Notes in Logic. Cambridge University Press, 2020.
		
		\bibitem{DMT}
		John Steel.
		\newblock The derived model theorem, 2008.
		
		\bibitem{Steel_outline}
		John~R. Steel.
		\newblock {\em An Outline of Inner Model Theory}, pages 1595--1684.
		\newblock Springer Netherlands, Dordrecht, 2010.
		
		\bibitem{Steel_projectively_well_ordered}
		J.R. Steel.
		\newblock Projectively well-ordered inner models.
		\newblock {\em Annals of Pure and Applied Logic}, 74(1):77--104, 1995.
		
	\end{thebibliography}
\end{document}